\documentclass[12pt,english,reqno]{amsart}

\usepackage{geometry}
\usepackage{amsmath,amssymb,amsthm,bbm,mathtools,comment}
\usepackage{amsfonts}
\usepackage[shortlabels]{enumitem}
\usepackage[pdftex,colorlinks,backref=page,citecolor=blue]{hyperref}
\usepackage[mathscr]{euscript}
\usepackage{tikz,babel,adjustbox}
\usepackage[numbers]{natbib}
\usepackage{graphicx}
\usepackage{caption}
\usepackage{subcaption}
\usepackage{verbatim}
\usepackage{array}
\usepackage{graphpap, pstricks}
\usepackage{etoolbox}
\usepackage{pifont}
\usepackage[final]{microtype}
\usepackage{cmtiup}
\usepackage[noabbrev,capitalize]{cleveref}

\hypersetup{pdfpagemode=UseNone,pdfstartview={XYZ null null 1.00}}
\usetikzlibrary{shapes.misc,calc,intersections,patterns,decorations.pathreplacing}
\usetikzlibrary{arrows,arrows.meta,shapes,positioning,decorations.markings}
\tikzstyle arrowstyle=[scale=1]

\allowdisplaybreaks

\makeatletter
\def\@settitle{\begin{center}%
  \bfseries\Large
  \@title
 \end{center}%
}
\patchcmd{\@setauthors}{\MakeUppercase}{\normalsize}{}{}
\makeatother

\theoremstyle{plain}
\newtheorem{thm}{Theorem}[section]
\newtheorem{lemma}[thm]{Lemma}

\newtheorem{cor}[thm]{Corollary}
\newtheorem{conj}[thm]{Conjecture}

\newtheorem{prob}[thm]{Problem}
\theoremstyle{remark}

\numberwithin{figure}{section}
\numberwithin{equation}{section}

\def\rr{\mathbb{R}}
\def\cA{\mathcal{A}} % Real numbers
\def\cB{\mathcal{B}} % Real numbers

\def\bY{\overline{Y}}
\def\bZ{\overline{Z}}
\DeclareMathOperator{\Cov}{\textnormal{Cov}} % Covariance
\DeclareMathOperator{\Eb}{\mathbb{E}} % Expectation
\DeclareMathOperator{\Pb}{\mathbb{P}} % Probability measure
\DeclareMathOperator{\cC}{\mathcal{C}} % Real numbers
\DeclareMathOperator{\one}{\mathbf{1}} % indicator function
\DeclareMathOperator{\cP}{\mathcal{P}} %
\DeclareMathOperator{\cS}{\mathcal{S}} %
\DeclareMathOperator{\SO}{\textnormal{SO}} % constants
\DeclareMathOperator{\sphere}{\mathbb{S}} % sphere
\DeclareMathOperator{\xb}{\mathbf{x}} % bold vector
\DeclareMathOperator{\Xb}{\mathbf{X}} % bold vector
\DeclareMathOperator{\Vb}{\mathbf{V}} % bold vector
\DeclareMathOperator{\Wb}{\mathbf{W}} % bold vector

\def\ap{p} % Probability
\def\aq{q} % Probability
\def\ar{r} % Probability

\definecolor{darkblue}{rgb}{0.0,0,0.7}
\newcommand{\darkblue}{\color{darkblue}}

\definecolor{darkred}{rgb}{0.68,0,0}

\definecolor{darkgreen}{rgb}{0,.38,0}

\newcommand{\defnb}[1]{\emph{\darkblue #1}}

\def\.{\hskip.06cm}
\def\ts{\hskip.03cm}

\begin{document}

\title{Unimodality for Radon partitions of random vectors}

\author{Swee Hong Chan}
\address{Department of Mathematics, Rutgers University, Piscataway, NJ 08854, USA}
\email{sweehong.chan@rutgers.edu}

\author{Gil Kalai}
\address{Einstein Institute of Mathematics, Hebrew University, Givat-Ram, Jerusalem 91904, Israel and Efi Arazi School of Computer Science, Reichman University, Herzliya 4610101, Israel}
\email{kalai@math.huji.ac.il}

\author{Bhargav Narayanan}
\address{Department of Mathematics, Rutgers University, Piscataway, NJ 08854, USA}
\email{narayanan@math.rutgers.edu}

\author{Natalya Ter-Saakov}
\address{Department of Mathematics, Rutgers University, Piscataway, NJ 08854, USA}
\email{nt399@rutgers.edu}

\author{Moshe White}
\address{Einstein Institute of Mathematics, Hebrew University, Givat-Ram, Jerusalem 91904, Israel}
\email{moshe.white@mail.huji.ac.il}

\date{1 July, 2025}

\begin{abstract}
	Consider the (almost surely) unique Radon partition of a set of $n$ random Gaussian vectors in $\mathbb R^{n-2}$; choose one of the two parts of this partition uniformly at random, and for $0 \le k \le n$,  let \( p_k \) denote the probability that it has size \( k \).  In this paper, we prove strong unimodality results for the distribution $(p_0,\dots,p_n)$.
%\com{SH}{I changed the definition of $p_k$ here, because the previous definition is not the same as what we studied.}
\end{abstract}
\maketitle

\section{Introduction} \label{s:intro}
In 1921, Radon proved~\citep{Rad} that any set $X$ of $n$ points in $\rr^{n-2}$ can be partitioned --- such a partition is called a \defnb{Radon partition of $X$} --- into two sets whose convex hulls intersect, and that if the points are in general position (as will be the case in what follows), then this partition is unique. Radon's theorem is a foundational result in convex geometry~\citep{Bar21,Eck79,Eck93}, and was, for instance, used by Radon to give one of the early proofs of Helly's theorem.

In this paper, we explore the behaviour of the Radon partition of a random set of $n$ points in $\rr^{n-2}$, specifically, of a set of $n$ independent Gaussian vectors. For such a random set $X$ and for $0 \le k \le n$, let $p_k = p_k^{(n)}$ denote the probability that a randomly chosen part of its (almost surely) unique Radon partition has $k$ elements, and note that $p_0=p_n=0$. Our goal in this paper is to understand this probability distribution $\mathbf{p}_n = (p_0,\dots, p_n)$, and to demonstrate in particular that $\mathbf{p}_n$ has some strong unimodality properties.

The motivation to study this distribution $\mathbf{p}_n$ is related to Sylvester's `four-point question' from 1864 that asks for the probability that four random planar points are in convex position. For Gaussian planar vectors, the answer --- $(p_2)^{(4)}$ in our notation --- was given by Maehara~\citep{Mae78}. The distribution $\mathbf{p}_n$ arises naturally when studying higher-dimensional analogues of Sylvester's question, and the problem of understanding this distribution was recently raised by Frick, Newman, and Pegden~\citep{FNP25} and by White~\citep{White} in precisely this context. While an exact description of the distribution $\mathbf{p}_n$ seems out of reach, we shall nevertheless establish a number of nontrivial properties of this distribution.

\subsection*{Main results}
Let $X = \{\Xb_1,\dots, \Xb_{n}\} \subset \rr^{n-2}$ be a set of $n$ independent standard normal random vectors. Note that these points are in general position almost surely, so they almost surely admit a unique Radon partition; we denote this partition by $\{\cA_n,\cB_n\}$. From these two sets, we choose one uniformly at random, which we call $\cS_n$, and we write $S_n=|\cS_n|$. For $n \geq 3$ and $0 \le k \le n$, let
\begin{equation}\label{eq:p_k}
	\ap_k = \ap_k^{(n)} = \Pb[S_n=k]
\end{equation}
be the probability that a randomly chosen part of the Radon partition of $X$ has size $k$. Writing $[n]$ for the set $\{1,\dots,n\}$, our main result is the following theorem.

\begin{thm}[Ultra log-concavity]\label{thm:ULC}
	For each $n \geq 3$, the distribution $\mathbf{p}_n = (p_0,\dots, p_n)$ is ultra log-concave, i.e., we have
	\[
		\left( \frac{\ap_{k}}{\binom{n}{k}} \right)^{2} \geq \frac{\ap_{k+1}}{\binom{n}{k+1}} \frac{\ap_{k-1}}{\binom{n}{k-1}}
	\]
	for each $k \in [n-1]$.
\end{thm}

We prove Theorem~\ref{thm:ULC} in Section~\ref{sec:proof-ULC}.
The following consequence of Theorem~\ref{thm:ULC} establishes a conjecture of White~\citep{White} asserting that the more balanced a partition is, the more likely it is to occur as the Radon partition of a Gaussian random set of points.

\begin{cor}[Unimodality]\label{cor:unimodality}
	For each $n \geq 3$, the distribution $\mathbf{p}_n = (p_0,\dots, p_n)$ is unimodal, i.e.,
	\[
		\ap_0 \leq \ap_1 \leq \dots \leq \ap_{\lfloor \frac{n}{2}\rfloor} = \ap_{\lceil \frac{n}{2}\rceil} \geq \dots \geq \ap_{n-1} \geq \ap_{n}. \]
\end{cor}

We prove Theorem~\ref{thm:ULC} by reducing it to a problem involving the log-concavity of sequences arising from the behaviour of percentiles of normal random variables. This is a problem of some independent interest and we shall say more about it below.
%\com{SH}{I change ``in the sequel'' to ``below'' to prevent misunderstanding that we are writing a sequel to this paper.}
We  also note that it remains an open problem to determine whether unimodality still holds when $\Xb_1,\ldots,\Xb_n$ are sampled uniformly from a bounded convex set $K$, see  Problem~\ref{prob:model-K}.

\subsection*{{Log-concavity of Youden's demon sequence}}
%\com{SH}{I change the macro from $\bar{Y}$ to $\overline{Y}$ because the line in the former is not visible for some pdf reader.}
Let $Y_1,\dots, Y_n$ be a set of $n$ independent standard normal random variables, and let $\bY = ({Y_1+\cdots+Y_n})/{n}$ denote their \defnb{sample mean}. For $k \in [n]$, we write
\begin{equation}\label{eq:q_k}
	\aq_k = \aq_k^{(n)} = \Pb \left [ \{Y_1,\dots, Y_k \leq \bY \leq Y_{k+1}, \dots, Y_{n} \}\right]
\end{equation}
for the probability that the sample mean is greater than the first $k$ samples and less than the last $n-k$ samples. The determination of these probabilities --- this is the problem of \defnb{Youden's demon} --- has been the subject of extensive work in mathematical statistics; see~\citep{You53,Ken,DmiGov97,FNP25}, for examples.

We shall deduce our main result, Theorem~\ref{thm:ULC}, from the following result about Youden's demon problem.
\begin{thm}[Log-concavity]\label{thm:LC}
	For each $n\ge 3$, the distribution $\mathbf{q}_n = (q_0,\dots, q_n)$ is log-concave, i.e., we have
	\begin{equation}\label{eq:LC}
		\aq_{k}^2 \geq \aq_{k+1} \aq_{k-1}
	\end{equation}
	for each $k \in [n-1]$.
\end{thm}

We prove Theorem~\ref{thm:LC} in Section~\ref{sec:proof-LC}.
We note that the assumption that $Y_1,\ldots, Y_n$ are normally distributed is essential here (as we shall see in Section~\ref{sec:counter-example}).
The reduction of Theorem~\ref{thm:ULC} to Theorem~\ref{thm:LC} (see Lemma~\ref{lem:iso}) more or less follows from the work of Frick, Newman, and Pegden~\citep{FNP25} and White~\citep{White} who, in turn, rely on a general result by Baryshnikov and Vitale~\citep{BV97}.

Let us also note that the assumption of normality cannot be dropped from Theorem~\ref{thm:LC}; we shall exhibit an i.i.d.\ collection of non-normal random variables for which the corresponding sequence $\aq_0,\dots, \aq_n$ is not unimodal, and therefore not log-concave.

\subsection*{Conditional associations}
We will prove Theorem~\ref{thm:LC} as a consequence of a more general phenomenon about \defnb{conditional associations}, but to state the precise result driving this, we need some notation.

Two events $A$ and $B$ in a probability space are said to be \defnb{positively associated} if $\Pb[A \cap B] \geq \Pb[A] \Pb[B]$, and we write $A \uparrow B$ to denote positive association. Analogously, we say that two real-valued random variables $X$ and $Y$ are \defnb{positively associated} if
\[
	\{X \geq s \} \uparrow \{Y \geq t \}
\]
for all $s,t \in \rr$, or equivalently if
\[
	\Eb[f(X) g(Y)] \geq \Eb[f(X)] \Eb[g(Y)]
\]
for all increasing functions $f,g: \rr \to \rr$; again, we write $X \uparrow Y$ to denote positive association (though we note that this differs slightly from the usage of this terminology in~\citep{KN}). Finally, for a collection of events $\cC$, we say that $X$ and $Y$ are \defnb{positively associated given $\cC$} if $X \uparrow Y$ conditional on any event in $\cC$, i.e., if
\[ \Pb \left[ X \geq s, Y \geq t \, \vert\, C \right ] \geq \Pb \left[ X \geq s \, \vert \, C \right ] \Pb \left[ Y \geq t \, \vert \, C \right ]\]
for all $s,t \in \rr$ and $C \in \cC$.

A sequence of random variables $Z_1,\dots, Z_m$ is \defnb{jointly normal} if there exist independent normal random variables $Y_1,\dots, Y_n$ such that each $Z_i$ is a linear combination of the $Y_j$'s. An important property that we will use throughout this paper is that for any pair of random variables $X, Y$ that is jointly normal, $X$ is independent of $Y$ if and only if $\Cov(X,Y)=\Eb[XY] - \Eb[X]\Eb[Y] = 0$; see e.g.~\citep{BH}, for a proof.
% \com{SH}{I change "for example" to "for a proof".}

We shall study jointly normal random variables $Z_1,\dots, Z_m$ with the following properties:
\begin{enumerate}[(A)]
	\item\label{eq:N} $\Eb[Z_i] = 0$ for all $i \in [m]$, $\Eb[Z_i Z_j ] \leq 0$ for all distinct $i,j \in [m]$, and
	\item\label{eq:P} $\Eb[Z_i (Z_1+\dots +Z_m)] \geq 0$ for all $i \in [m]$.
\end{enumerate}
Note that~\ref{eq:N} says that the covariance matrix of $Z_1,\dots, Z_m$ is a \defnb{Stieltjes matrix} (see~\citep{KR}, for more details on this subject), while~\ref{eq:P} tells us that each $Z_i$ has nonnegative covariance with the sample mean $\bZ$. We call such a collection of random variables a \defnb{repulsive–cooperative Gaussian ensemble}.

With the convention that $-\infty \le a \le b \le \infty$ whenever we refer to a closed interval $[a,b]$, we have the following theorem about conditional associations of repulsive–cooperative Gaussian ensembles.

\begin{thm}\label{thm:CNA}
	Let $Z_1,\dots, Z_{m}$ be a repulsive–cooperative Gaussian ensemble, and let $\cC$ be the collection of events of the form
	\begin{equation}\label{eq:cond}
		\bigcap_{i\in[m]} \left\{ Z_i \in [a_i,b_i] \right \}.
	\end{equation}
	Then, for all $i \in [m]$, $\bZ$ and $Z_i$ are positively associated given $\cC$.
\end{thm}

We prove Theorem~\ref{thm:CNA} in Section~\ref{sec:proof-CNA}.
Let us note that the positive association of $\bZ$ and $Z_i$ without conditioning is a consequence of~\ref{eq:P}; see~\citep{Pitt} for a proof. However, conditioning on $\cC$ does not in general preserve positive associations (as we shall see in Section~\ref{sec:counter-example}), so we need a more involved argument that exploits~\ref{eq:N}.

\subsection*{Organisation}
Further background that provides additional motivation for the questions studied in this paper is provided in Section~\ref{sec:background}. We prove Theorem~\ref{thm:CNA} in Section~\ref{sec:proof-CNA}. We then use Theorem~\ref{thm:CNA} to prove Theorem~\ref{thm:LC} in Section~\ref{sec:proof-LC}. Theorem~\ref{thm:ULC} and Corollary~\ref{cor:unimodality} are then deduced from Theorem~\ref{thm:LC} in Section~\ref{sec:proof-ULC}. In Section~\ref{sec:counter-example}, we demonstrate that Theorem~\ref{thm:LC} requires the normality assumption, and that the assumptions in Theorem~\ref{thm:CNA} are necessary as well. Finally, in Section~\ref{sec:sto-conv}, we discuss several problems, results, and empirical findings on related topics.

\section{Background}\label{sec:background}

In this section, we provide some background and historical context for the various problems discussed in Section~\ref{s:intro}.

\subsection*{Around Radon's theorem}
Recall that Radon's theorem states that any set of $n$ points in $\rr^{d}$, where $n\ge d+2$, can be partitioned into two sets whose convex hulls intersect. This theorem was first proved by Radon~\citep{Rad} in 1921, and it has since become a cornerstone of convex geometry. Radon's theorem was used by Radon to give a proof of Helly's theorem that asserts that if a finite family of convex sets in $\rr^{d}$ (with at least $d+1$ members) has the property that every $d+1$ sets in the family have a point in common, then there is a point in common to all sets in the family. In 1966, Tverberg~\citep{Tve66} published a far-reaching extension of Radon's theorem conjectured a few years earlier by Birch.

\begin{thm}\label {thm:Tve}
	Any $(r-1)(d+1)+1$ points in $\mathbb R^d$ can be partitioned into $r$ sets whose convex hulls intersect.
\end{thm}
Any partition that satisfies the conditions of Theorem~\ref{thm:Tve} is referred to as a \defnb{Tverberg partition of order $r$}. It would be interesting to investigate analogues of our results for Tverberg partitions in place of Radon partitions, a topic about which we will say more later in Section~\ref{sec:sto-conv}. For more background on these results, see~\citep{Bar21,Eck79,Eck93,Rou01,BS18,BK22}.

\subsection*{Around Sylvester's four-point question}
In 1864, Sylvester asked for the probability that four random points in the plane are in convex position~\citep{Syl64}. He proposed that the answer is 1/4 if the points are ``taken at random in an indefinite plane", and also asked what happens when the points are taken at random from a set $K \subset \rr^2$. When the points are taken at random from the Gaussian distribution, then this probability is equal to $p_2^{(4)}$ (as defined in~\eqref{eq:p_k}), and understanding the distribution of the Radon partition of $n$ random points in $\mathbb R^{n-2}$ is a natural higher-dimensional extension of Sylvester's problem. For more on Sylvester's four-point problem, involving extensions both to more points and to higher dimensions, see~\citep{Syl64,Pfi89,Bar99,Val95}.

For four Gaussian points in the plane, Sylvester's problem has a nice answer going back to Maehara~\citep{Mae78} (see also~\citep{Bla08}).
\begin{thm}
The probability that four random Gaussian points in $\mathbb R^2$ are in convex position is
\[\frac {6}{\pi}\arcsin (\frac {1}{3} )\approx 0.649 .\]
\end {thm}
The study of Radon partitions of a set of higher-dimensional Gaussian vectors is a natural extension of Sylvester's question that was recently studied by Frick, Newman, and Pegden~\citep{FNP25} and by White~\citep{White}. For more historical context and relevant background, see White's thesis~\citep{White} and the references therein.

\subsection*{Around unimodality and log-concavity} Unimodality is a very natural property that mathematicians perceive in a number of places, leading to a surge of conjectures about settings where this property may arise. In trying to demonstrate unimodality, it is quite natural to (attempt to) demonstrate log-concavity instead, since it is the more robust of the two properties. For instance, the convolution of two log-concave sequences remains log-concave, whereas unimodality is not preserved under such operations; see~\citep{Lig}, for example. In this vein, our proof (Corollary~\ref{cor:unimodality}) of White's conjecture through Theorem~\ref{thm:ULC} adheres to this general approach.

We point the reader to the surveys~\citep{Bre1,Bre2,Sta} for an overview of classical unimodality and
log-concavity results. It is worth mentioning that more recently, Huh and his collaborators have advanced an algebraic-geometric approach that has resolved a number of longstanding conjectures about log-concavity; see~\cite{Huh} for a survey,
and~\cite{CP} for a more combinatorial perspective on this direction. However, our proof of Theorem~\ref{thm:ULC} relies on a different approach, and is inspired by the work such as~\cite{Pem,BBL,KN} that link log-concavity with negative association.

\subsection*{Around central limit theorems}\label{ss:mountain} One consequence of Theorem~\ref{thm:ULC} is that the scaling limit of the distributions $\mathbf{p}_n = ( p_0, \dots, p_n )$ as $n \to \infty$, if it exists, must be a log-concave distribution. It is therefore natural to inquire what this limiting distribution is; this is addressed by the following central limit theorem.

\begin{thm}[Central limit theorem]\label{thm:CLT}
	As $n \to \infty$, we have
	\[ \frac{S_n- n/2}{\sigma \sqrt{n}} \xrightarrow[\mathrm{dist}]{} \mathcal{N}(0,1),\]
	where $S_n$ is sampled from $\mathbf{p}_n$, and $\sigma^2 = {1}/{4}-{1}/{(2\pi)}$.
\end{thm}

It is possible to prove Theorem~\ref{thm:CLT} by reducing the problem to an analogous central limit theorem for Youden's demon, a version of which was established by Dmitrienko and Govindarajulu~\citep[Thm3.1]{DmiGov97} using the Bahadur representation of quantiles (see, e.g., \citep{Ghosh}).
To keep the exposition self-contained, we include a more direct proof of Theorem \ref{thm:CLT} in Appendix \ref{sec:proof-CLT}, using the second-moment method.

% expressing $S_n$ as a sum of weakly dependent random variables and then approximating the moment sequence of the $S_n$'s to show that this sequence converges to the moment sequence of a normal random variable, though one needs some care in doing this; indeed, none of the standard central limit theorems that we know of seem to apply in this situation. However, the analogous central limit theorem for Youden's demon problem is known to be true, see~\citep[Thm~3.1]{DmiGov97}, and Theorem~\ref{thm:CLT} is an immediate consequence.

\section{Conditional associations}\label{sec:proof-CNA}
Our goal in this section is to prove Theorem~\ref{thm:CNA}. We start with the following standard lemma; we include the proof for completeness.

\begin{lemma}\label{lem:coupling}
	Let $W$ be a real-valued random variable, and let $[a,b]$ and $[c,d]$ be intervals of $\rr$ such that $a<c$ and $b<d$.
	Let $U$ and $V$ be random variables such that
	\begin{align}\label{eq:coupling-0}
		\begin{split}
			 & \text{$U$ is equal in distribution to $W$ conditioned on $W \in [a,b]$, and} \\
			 & \text{$V$ is equal in distribution to $W$ conditioned on $W \in [c,d]$.}
		\end{split}
	\end{align}
	Then there exists a coupling of $U$ and $V$ such that $U\leq V$. That is, there exists a probability space $\Omega$, a probability measure $\mu: \Omega \to [0,1]$, and measurable functions $U,V: \Omega \to \rr $ such that $U$ and $V$ satisfy~\eqref{eq:coupling-0} and $U(\omega) \leq V(\omega)$ for all $\omega \in \Omega$.
\end{lemma}
\begin{proof}
	The lemma is straightforward if $b\leq c$, so we assume that $b>c$. It suffices to show, for all $s \in \rr$, that
	\[
		\Pb\left[ W \geq s \mid W \in [a,b] \right] \leq \Pb\left[ W \geq s \mid W \in [c,d] \right].
	\]

	Let $W'$ be the random variable $W$ conditioned on $W \in [a,b]$. Notice that the events $\{W' \geq s \}$ and $\{W' \geq c \}$ are positively associated. Indeed, we clearly have
	\begin{align}\label{eq:coupling-1.5}
		\Pb\left[ W' \geq s, W' \geq c \right] = \Pb\left[ W' \geq \max \{s,c\} \right] \geq \Pb\left[ W' \geq s \right] \Pb\left[ W' \geq c \right].
	\end{align}
	This implies that
	\begin{equation}\label{eq:coupling-1}
		\begin{split}
			\Pb\left[ W \geq s \mid W \in [c,b] \right] & = \Pb\left[ W' \geq s \mid W' \geq c\right]                                      \\
			                                            & \geq \Pb\left [ W' \geq s \right] = \Pb\left[ W \geq s \mid W \in [a,b] \right].
		\end{split}
	\end{equation}
	Now let $W''$ be the random variable $W$ conditioned on $W \in [c,d]$. As before, the events $\{W'' \geq s \}$ and $\{W''> b \}$ are positively associated, whence the events $\{W'' \geq s \}$ and $\{W'' \leq b \}$ are negatively associated. It then follows that
	\begin{equation}\label{eq:coupling-2}
		\begin{split}
			\Pb\left[ W \geq s \mid W \in [c,b] \right] & = \Pb\left[ W'' \geq s \mid W'' \leq b \right]                                   \\
			                                            & \leq \Pb\left[ W'' \geq s \right] = \Pb\left[ W \geq s \mid W \in [c,d] \right].
		\end{split}
	\end{equation}
	The lemma now follows by combining~\eqref{eq:coupling-1} and~\eqref{eq:coupling-2}.
\end{proof}

Let 
 $Z_1,\dots, Z_m$  be a repulsive-cooperative Gaussian ensemble (as defined by \ref{eq:N} and \ref{eq:P})  with sample mean $\bZ$.
 Let $f:[-\infty,\infty]^{2m+1} \to \rr$ be the function defined by
\begin{equation}\label{eq:deff}
	f(s, a_1,b_1,\dots, a_m,b_m) = \Pb\left[ \bZ \geq -s \, \Big \vert \, \bigcap_{i \in [m]} \{Z_i \in [a_i,b_i]\} \right].
\end{equation}
We shall use Lemma~\ref{lem:coupling} to prove the following lemma that will in turn imply Theorem~\ref{thm:CNA}.

\begin{lemma}\label{lem:inc}
	Let $Z_1,\dots, Z_{m}$ be a repulsive–cooperative Gaussian ensemble. Then $f$ is an increasing function, i.e.
	\[ f(s,a_1,b_1,\dots, a_m, b_m) \leq f(t,c_1,d_1,\dots, c_m, d_m)\]
	whenever $s \leq t$, and $a_i \leq c_i$ and $b_i \leq d_i$ for all $i \in [m]$.
\end{lemma}

\begin{proof}
	We prove the claim by induction on $m$. First, let $m=1$, and let $s$, $t$, $a_1$, $b_1$, $c_1$ and $d_1$ be real numbers satisfying $s \leq t$, $a_1\leq c_1$ and $b_1 \leq d_1$. Let $U$ be the random variable $Z_1$ conditioned on $Z_1 \in [a_1,b_1]$, and let $V$ be the random variable $Z_1$ conditioned on $Z_1 \in [c_1,d_1]$. We then have
	\[
		f(s,a_1,b_1) = \Pb\left[ U \geq -s \right] \leq \Pb\left[ V \geq -s \right] \leq 	\Pb\left[ V \geq -t \right] = f(t,c_1,d_1),
	\]
	where the first inequality follows from Lemma~\ref{lem:coupling}.

	Now, let $m \geq 2$. It is clear that $f$ is increasing in the first variable $s$. By symmetry, it therefore suffices to show that
	\[ f(s,a_1,b_1,\dots, a_{m-1}, b_{m-1}, {\color{blue}{a_m}}, {\color{blue}{b_m}}) \leq f(s,a_1,b_1,\dots, a_{m-1}, b_{m-1}, {\color{blue}{c_m}}, {\color{blue}{d_m}}),\]
	whenever $a_m \leq c_m$ and $b_m \leq d_m$.

	Let $Z_1',\dots, Z_{m-1}'$ be jointly normal random variables given by
	\begin{equation}\label{eq:Y'}
		Z_i' = Z_i + \alpha_i Z_m,
	\end{equation}
	where $\alpha_i$ is the nonnegative real number given by
	\[ \alpha_i = - \frac{\Eb[Z_i Z_m]}{\Eb[Z_m^2]}; \]
	note that $\alpha_i$ is nonnegative because $Z_1,\dots, Z_m$ satisfy~\ref{eq:N}. It follows that $\Eb[Z_i' Z_m]=0$ for $i\in [m-1]$, so $Z_1',\dots, Z'_{m-1}$ are independent of $Z_m$. Also, note that the sample mean $\overline{Z'}$ of $Z_{1}',\dots, Z_{m-1}'$ satisfies
	\begin{equation}\label{eq:bY'}
		\overline{Z'} = \tfrac{m}{m-1} \overline{Z} - \tfrac{m}{m-1} \beta Z_m,
	\end{equation}
	where $\beta$ is the real number
	\[ \beta = \frac{\Eb[Z_m(Z_1+\dots+Z_m)]}{m \Eb[Z_m^2]}, \]
	which is again nonnegative because $Z_1,\dots, Z_m$ satisfy~\ref{eq:P}.

	Now, for distinct $i,j \in [m-1]$, we have
	\begin{align*}
		\Eb[Z_i' Z_j'] & = \Eb[Z_iZ_j] + \alpha_i \Eb[Z_j Z_m] + \alpha_j \Eb[Z_i Z_m] + \alpha_i \alpha_j \Eb[Z_m^2] \\
		               & = \Eb[Z_iZ_j] - \frac{\Eb[Z_i Z_m] \Eb[Z_j Z_m]}{\Eb[Z_m^2]} \leq 0,
	\end{align*}
	where the last inequality is because $Z_1,\dots, Z_m$ satisfy~\ref{eq:N}. Hence $Z_1',\dots, Z_{m-1}'$ satisfy~\ref{eq:N}. Next, for $i \in [m-1]$, we have
	\begin{align*}
		\Eb[Z_i'(Z_1'+\dots+Z_{m-1}')] & = \Eb[Z_i'(Z_1+\dots+Z_{m-1})]                                                             \\
		                               & = \Eb[Z_i(Z_1+\dots+Z_{m-1})] - \frac{\Eb[Z_iZ_m] \Eb[Z_m(Z_1+\dots+Z_{m-1})]}{\Eb[Z_m^2]} \\
		                               & = \Eb[Z_i(Z_1+\dots+Z_m)] - \frac{\Eb[Z_iZ_m] \Eb[Z_m(Z_1+\dots+Z_m)]}{\Eb[Z_m^2]}         \\
		                               & \geq 0,
	\end{align*}
	where the last inequality holds because $Z_1,\dots, Z_m$ satisfy~\ref{eq:N} and~\ref{eq:P}. It follows that $Z_1',\dots, Z_{m-1}'$ satisfy~\ref{eq:P} as well, so $Z_1',\dots, Z_{m-1}'$ constitute a repulsive-cooperative Gaussian ensemble as well.

	Recall that $f(s,a_1,b_1,\dots, a_{m-1}, b_{m-1}, {\color{blue}{a_m}}, {\color{blue}{b_m}})$ is equal to
	\begin{equation*}
		\Pb\left[ \bZ \geq -s \, \Big \vert \, \bigcap_{i \in [m-1]} \{Z_i \in [a_i,b_i]\} \cap {\color{blue}{\{Z_m \in [a_m,b_m]\} }} \right].
	\end{equation*}
	From~\eqref{eq:Y'} and~\eqref{eq:bY'}, this is equal to
	\begin{equation*}
		\Pb\left[ \bar{Z'} \geq \tfrac{-m}{m-1}(s+\beta Z_m) \, \Big \vert \, \bigcap_{i \in [m-1]} \{Z_i' \in [a_i+\alpha_i Z_m,b_i+\alpha_i Z_m]\} \cap {\color{blue}{\{Z_m \in [a_m,b_m]\} }} \right].
	\end{equation*}
	By the same argument, $f(s,a_1,b_1,\dots, a_{m-1}, b_{m-1}, {\color{blue}{c_m}}, {\color{blue}{d_m}})$ is equal to
	\begin{equation*}
		\Pb\left[ \bar{Z'} \geq \tfrac{-m}{m-1}(s+\beta Z_m) \, \Big \vert \, \bigcap_{i \in [m-1]} \{Z_i' \in [a_i+\alpha_i Z_m,b_i+\alpha_i Z_m]\} \cap {\color{blue}{\{Z_m \in [c_m,d_m]\} }} \right].
	\end{equation*}

	Now, let $W$ be the random variable defined by
	\[ W = Z_m \text{ conditioned on } \bigcap_{i \in [m-1]} \{Z_i' \in [a_i+\alpha_i Z_m,b_i+\alpha_i Z_m]\},\]
	and let $U$ and $V$ be the random variables given by
	\begin{align}\label{eq:coupling-A}
		\begin{split}
			 & \text{$U$ is equal in distribution to $W$ conditioned on $W \in [a_m,b_m]$, and} \\
			 & \text{$V$ is equal in distribution to $W$ conditioned on $W \in [c_m,d_m]$.}
		\end{split}
	\end{align}
	It then follows from Lemma~\ref{lem:coupling} that there exists a coupling of $U$ and $V$ such that $U\leq V$, that is, there exists a probability space $\Omega$, a probability measure $\mu: \Omega \to \rr$, and measurable functions $U,V: \Omega \to \rr$ such that $U,V$ satisfy~\eqref{eq:coupling-A}, and  $U(\omega) \leq V(\omega)$ for all $\omega \in \Omega$. By the law of total probability, we see that $f(s,a_1,b_1,\dots, a_{m-1}, b_{m-1}, {\color{blue}{a_m}}, {\color{blue}{b_m}})$ is equal to
	\begin{align*}
		\int_{\Omega} \Pb\left[ \bar{Z'} \geq \tfrac{-m}{m-1}(s+\beta {\color{blue}{U(\omega)}}) \,\Big \vert \, \bigcap_{i \in [m-1]} \{Z_i' \in [a_i+\alpha_i {\color{blue}{U(\omega)}} ,b_i+\alpha_i {\color{blue}{U(\omega)}}] \} \right] d\mu(\omega).
	\end{align*}
	Now, let $g:[-\infty,\infty]^{2m-1}\to \rr$ be the function defined by
	\[ g(s,a_1,b_1,\dots, a_{m-1}, b_{m-1}) = \Pb\left[ \bar{Z'} \geq -s \, \Big \vert \, \bigcap_{i \in [m-1]} \{Z'_i \in [a_i,b_i]\} \right], \]
	and note that $g$ is an increasing function by the induction hypothesis. It follows that $f(s,a_1,b_1,\dots, a_{m-1}, b_{m-1}, {\color{blue}{a_m}}, {\color{blue}{b_m}})$ is equal to
	\begin{align}\label{eq:beta-1}
		\int_{\Omega}	g\left(\tfrac{m}{m-1}(s+\beta{\color{blue}{U(\omega)}}), (a_i+\alpha_i {\color{blue}{U(\omega)}} ,b_i+\alpha_i {\color{blue}{U(\omega)}})_{i \in [m-1]} \right)  d\mu(\omega).
	\end{align}
	Similarly, $f(s,a_1,b_1,\dots, a_{m-1}, b_{m-1}, {\color{blue}{c_m}}, {\color{blue}{d_m}})$ is equal to
	\begin{align}\label{eq:beta-2}
		\int_{\Omega}	g\left(\tfrac{m}{m-1}(s+\beta{\color{blue}{V(\omega)}}), (a_i+\alpha_i {\color{blue}{V(\omega)}} ,b_i+\alpha_i {\color{blue}{V(\omega)}})_{i \in [m-1]} \right) d\mu(\omega).
	\end{align}
	Since $g$ is an increasing function and $U(\omega) \leq V(\omega)$ by the chosen coupling, we conclude that
	\[ f(s,a_1,b_1,\dots, a_{m-1}, b_{m-1}, {\color{blue}{a_m}}, {\color{blue}{b_m}}) \leq f(s,a_1,b_1,\dots, a_{m-1}, b_{m-1}, {\color{blue}{c_m}}, {\color{blue}{d_m}}),\]
	completing the proof.
\end{proof}

With Lemma~\ref{lem:inc} in hand, we are now ready to prove Theorem~\ref{thm:CNA}.

\begin{proof}[Proof of Theorem~\ref{thm:CNA}] We assume without loss of generality that $i=1$. It suffices to show, for any $s,t \in \rr$, that
	\[ \Pb\left[ \bZ \geq s \,\vert\, \{Z_1 \geq t\} \cap C \right] \geq \Pb\left[ \bZ \geq s \,\vert\, C \right], \]
	where $C$ is any event of the form
	\[ \bigcap_{i \in[m]} \{ Z_i \in [a_i,b_i] \}. \]

	With $f$ defined as in~\eqref{eq:deff}, we see that
	\begin{align*}
		\Pb\left[ \bZ \geq s \,\vert\, \{Z_1 \geq t\} \cap C \right] & = f(-s,\max\{t,a_1\}, b_1, a_2,b_2,\dots, a_m,b_m), \text{ and } \\
		\Pb\left[ \bZ \geq s \,\vert\, C \right]                     & = f(-s,a_1, b_1,\dots, a_m,b_m);
	\end{align*}
	the claim now follows from Lemma~\ref{lem:inc}, i.e., the fact that $f$ is an increasing function.
\end{proof}

\section{Log-concavity for Youden's demon}\label{sec:proof-LC}
Our goal in this section is to prove Theorem~\ref{thm:LC}. We start with the following lemma, which is a direct consequence of Theorem~\ref{thm:CNA}. Let $Z_1,\dots, Z_m$ be a sequence of jointly normal random variables and $\bZ$ be the corresponding sample mean. Writing
\[ \ar_k(Z_1,\dots,Z_m) = \Pb \left[ \bZ \geq 0 \,\vert\, \{Z_1,\dots, Z_k \geq 0, Z_{k+1},\dots, Z_m \leq 0\} \right], \]
we have the following lemma.

\begin{lemma}\label{lem:red-Y}
	Let $Z_1,\dots, Z_m$ be a repulsive–cooperative Gaussian ensemble. Then, for all $k \in \{0,\dots,m-1\}$, we have
	\[ \ar_{k+1}(Z_1,\dots, Z_m) \geq \ar_{k}(Z_1,\dots, Z_m). \]
\end{lemma}

Notice that the statement of Lemma~\ref{lem:red-Y} is essentially trivial if $Z_1,\dots, Z_m$ are independent. However, we only know that these random variables have negative covariances (from~\ref{eq:N}), so extra arguments are needed.

\begin{proof}[Proof of Lemma~\ref{lem:red-Y}]
	Let $C$ be the event
	\[ C = \{ Z_1,\dots, Z_k \geq 0, Z_{k+2}, \dots, Z_m \leq 0 \}. \]
	Note that this event is of a form that is admissible in Theorem~\ref{thm:CNA}; indeed, in the language of Theorem~\ref{thm:CNA}, we may take $[a_i,b_i]=[0,\infty]$ for $i \in [k]$, $[a_{k+1},b_{k+1}]=[-\infty,\infty]$, and $[a_i,b_i]=[-\infty,0]$ for $i \in \{k+2,\dots, m\}$.

	Now, by Theorem~\ref{thm:CNA}, we know that $\bZ$ and $Z_{k+1}$ are positively associated given $C$. Consequently, we have
	\begin{align*}
		\ar_k(Z_1,\dots, Z_m) & = \Pb\left[ \bZ \geq 0 \, \vert \, \{{\color{blue}{Z_{k+1}\leq 0}} \} \cap C\right] \leq \Pb\left[ \bZ \geq 0 \,\vert\, C\right] \\
		                      & \leq \Pb\left[ \bZ \geq 0 \, \vert \, \{{\color{blue}{Z_{k+1}\geq 0}} \} \cap C\right]= \ar_{k+1}(Z_1,\dots, Z_m),
	\end{align*}
	and the result follows.
\end{proof}

We are now ready to prove Theorem~\ref{thm:LC}.

\begin{proof}[Proof of Theorem~\ref{thm:LC}]
	Let $Y_1,\dots, Y_n$ be a set of independent standard normal random variables, and let $1 \leq k \leq n-1$. We need to show that
	\[\aq_{k}^2 \geq \aq_{k+1} \aq_{k-1},\]
	where $\aq_k = \Pb[ Y_1,\dots, Y_k \leq \bY \leq Y_{k+1}, \dots, Y_{n}]$, and $q_0 = q_n = 0$. We may assume that $2 \le k \leq n-2$, as otherwise, the right hand side of the bound above is $0$ and there is nothing to prove. Let $m= n-1$, and let $Z_1,\dots, Z_m$ be jointly normal random variables defined by
	\[ Z_i = \
		\begin{cases}
			\bY - Y_i     & \text{ if $1 \leq i\leq k$},    \\
			\bY - Y_{i+1} & \text{ if $k+1 \leq i \leq m$}.
		\end{cases}
	\]
	Direct calculation shows that
	\begin{align*}
		 & \Eb[Z_i^2] = 1 - \frac{1}{n},            \\
		 & \Eb[Z_i Z_j] = -\frac{1}{n}, \text{ and} \\
		 & \Eb[Z_i(Z_1+\dots+Z_m)] = \frac{1}{n};
	\end{align*}
	it follows that $Z_1,\dots, Z_m$ is a repulsive–cooperative Gaussian ensemble.

	Since
	\[ Z_1+\dots + Z_m = Y_{k+1}- \bY,\]
	we have
	\begin{align*}
		\aq_k & = \Pb\left[ \{Y_{k+1} \geq \bY \} \cap \{Y_1,\dots, Y_k \leq \bY \leq Y_{k+2}, \dots, Y_{n} \} \right]                 \\
		      & = \Pb\left[ \{Z_1+\dots +Z_m {\color{blue}\geq 0} \} \cap \{Z_1,\dots, Z_k \geq 0 \geq Z_{k+1}, \dots, Z_m \} \right].
	\end{align*}
	By the same reasoning,
	\[
		\aq_{k+1} = \Pb\big[ Z_1+\dots +Z_m {\color{blue}\leq 0}, Z_1,\dots, Z_k \geq 0, Z_{k+1}, \dots, Z_m \leq 0 \big].
	\]
	It then follows that
	\begin{align*}
		\frac{\aq_k}{\aq_k + \aq_{k+1}} & = \Pb \left[ Z_1+\dots +Z_m \geq 0 \,\vert\, Z_1,\dots, Z_k \geq 0\geq Z_{k+1},\dots, Z_m \ \right] \\
		                                & = \ar_k(Z_1,\dots, Z_m).
	\end{align*}

	Now, let $Z_1',\dots, Z_m'$ be random variables defined by
	\[ Z_i' := \
		\begin{cases}
			\bY - Y_i     & \text{ if $1 \leq i\leq k-1$}, \\
			\bY - Y_{i+1} & \text{ if $k \leq i \leq m$}.
		\end{cases}
	\]
	Note that $(Z_1',\dots, Z_m')$ and $(Z_1,\dots, Z_m)$ are identically distributed, and as before, we have
	\begin{align*}
		\frac{\aq_{k-1}}{\aq_{k-1} + \aq_{k}} & = \Pb \left[ Z'_1+\dots +Z'_m \geq 0 \,\vert\, Z'_1,\dots, Z'_{k-1} \geq 0\geq Z'_k,\dots, Z'_m \ \right] \\
		                                      & = \ar_{k-1}(Z'_1,\dots, Z'_m).
	\end{align*}

	Lemma~\ref{lem:red-Y} then tells us that
	\begin{align*}
		\frac{\aq_k}{\aq_{k} + \aq_{k+1}} = \ar_k(Z_1,\dots, Z_m) \ge \ar_{k-1}(Z_1,\dots, Z_m) = \ar_{k-1}(Z_1',\dots, Z_m') = \frac{\aq_{k-1}}{\aq_{k-1} + \aq_{k}},
	\end{align*}
	and it is easy to check that the inequality above is equivalent to Theorem~\ref{thm:LC}, completing the proof.
\end{proof}

\section{Ultra log-concavity for Radon partitions}\label{sec:proof-ULC}
For $n \ge 3$, recall that $\cS_n$ is a uniformly random part of the (almost surely) unique Radon partition of $X = \{\Xb_1,\dots, \Xb_{n}\} \subset \rr^{n-2}$, a set of $n$ independent standard normal random vectors. For a set $Y_1,\dots, Y_n$ of independent standard normal random variables, let $\cP_n \subseteq [n]$ be defined by
\[\cP_n = \{ i \in [n] \mid Y_i \leq \bY \}; \]
i.e., the $\cP_n$ is the (random) set of indices $i$ for which $Y_i$ is less than the sample mean. The next lemma follows from an argument analogous to those in~\citep{BV97,FNP25}, but we include a proof for completeness.

\begin{lemma}\label{lem:iso}
	The random variables $\cS_n$ 	and $\cP_n$ are equal in distribution.
\end{lemma}

\begin{proof}
	Let $K$ be the subset of the unit sphere $\sphere^{n-1}$ defined by
	\[ K = \left \{ \xb \in \sphere^{n-1} : \one \cdot \xb=0 \right \}, \]
	where $\one=(1,\dots, 1) \in \rr^n$, and let $\mu$ be the uniform probability measure on $K$. Note that this measure is invariant under the action of the group $G \leq \SO(n)$ given by
	\[ G = \left \{ A \in \SO(n):  \one \cdot A \in \{\one,-\one\} \right\}, \]
    where $\SO(n)$ is the special orthogonal group.
	It is easily seen that the action of $G$ on $K$ is transitive, so $\mu$ is the unique probability measure on $K$ that is invariant under the action of $G$.

	We now express the law of $\cS_n$ in terms of $\mu$. Let $\Vb=(V_1,\dots, V_{n}) \in K$ be a vector that satisfies
	\[
		V_1 \Xb_1 + \cdots + V_{n} \Xb_{n} = 0.
	\]
	Note that distinct $i,j \in [n]$ belong to the same part of the Radon partition of $\Xb_1,\dots, \Xb_{n}$ if and only if $V_i$ and $V_j$ have the same sign. Also note that $\Vb$ is unique up to a sign change almost surely since $\Xb_1,\dots, \Xb_{n}$ are in general position almost surely. From the two vectors $\Vb$ and $-\Vb$, we choose one uniformly at random, and we denote this random vector by $\Wb=(W_1,\dots, W_{n})$. It then follows from the construction that
	\begin{equation}\label{eq:iso-1}
		\cS_n = \{ i \in [n] \mid W_i \geq 0 \}.
	\end{equation}
	Now, note that the measure induced by $(\Xb_1,\dots, \Xb_{n})$ is invariant under the action of $\SO(n)$ as $\Xb_1,\dots, \Xb_n$ are independent standard normal random vectors. It then follows that the measure induced by $\Wb$ is also invariant under the action of $G$. Therefore, we conclude that $\Wb$ is distributed according to $\mu$.

	Next, we express the law of $\cP_n$ in terms of $\mu$. Let $\Vb'=(V_1',\dots, V_n') \in \rr^n$ be the random vector given by
	\[ V_i' = \bY - Y_i \]
	for $i \in [n]$,  and let $\Wb'=(W_1',\dots, W_n')$ be the vector in $K$ given by
	\[ \Wb' = \frac{1}{\sqrt{\Vb' \cdot \Vb'}} \Vb'. \]
	It then follows from the construction that
	\begin{equation}\label{eq:iso-2}
		\cP_n = \{ i \in [n] :  V_i' \geq 0 \} = \{ i \in [n] :  W_i' \geq 0 \}.
	\end{equation}
	Since the measure induced by $(Y_1,\dots, Y_n)$ is invariant under the action of $\SO(n)$, it follows that the measure induced by $\Wb'$ is unique under the action of $G$. We conclude that $\Wb'$ is distributed according to $\mu$ as well, and the lemma now follows from comparing~\eqref{eq:iso-1} and~\eqref{eq:iso-2}.
\end{proof}

We may now deduce Theorem~\ref{thm:ULC} from Theorem~\ref{thm:LC} and Lemma~\ref{lem:iso}.
\begin{proof}[Proof of Theorem~\ref{thm:ULC} ]
	It follows from Lemma~\ref{lem:iso} that, for $k \in \{0,\dots,n\}$, we have
	\begin{align*}
		\ap_k & = \Pb[|\cS_n|=k] = \Pb[|\cP_n|=k]                                                               \\
		      & = \binom{n}{k} \Pb[Y_1, \dots, Y_k \leq \bY, Y_{k+1},\dots, Y_n \geq \bY] = \binom{n}{k} \aq_k,
	\end{align*}
	where $\aq_k$ is as in~\eqref{eq:q_k}. The result now follows from Theorem~\ref{thm:LC}.
\end{proof}

Unimodality is of course a straightforward consequence of Theorem~\ref{thm:ULC}.
\begin{proof}[Proof of Corollary~\ref{cor:unimodality}]
	It is clear that $(\ap_{0},\dots, \ap_{n})$ is a sequence of non-negative integers with no internal zeroes. It is ultra log-concave by Theorem~\ref{thm:ULC}, whence it is  log-concave as well. This in turn implies that the sequence is unimodal, i.e., that $\ap_0 \leq \cdots \leq \ap_{\ell} \geq \cdots \geq \ap_{n}$ for some $\ell \in \{0,\dots,n\}$. On the other hand, we also know that $\ap_{k}=\ap_{n-k}$ for all $k \in \{0,\dots,n\}$ by the symmetry in the construction of $\cS_n$; the corollary follows.
\end{proof}

\section{Necessity of assumptions}\label{sec:counter-example}
In this section, we shall show that the assumptions of Theorems~\ref{thm:LC} and~\ref{thm:CNA} are necessary.

First, we show that the conclusions of Theorem~\ref{thm:LC} need not hold without the assumption of normality. Below, we construct independent and identically distributed non-normal random variables that do not satisfy the conclusion of Theorem~\ref{thm:LC}.

Let $A_1, \dots, A_n$ be i.i.d.\ random variables with law given by
\[ \Pb[A_i =x ] =
	\
	\begin{cases}
		\tfrac{1}{n}   & x \in \{-1, 1\}, \text{ and} \\
		1-\tfrac{2}{n} & x = 0.
	\end{cases}
\]
Let $B_1,\dots, B_n$ be i.i.d.\ random variables which are uniform on the interval $[-\tfrac{1}{e^n}, \tfrac{1}{e^n}]$, and are independent of $A_1,\dots, A_n$.
For $i \in [n]$, let
\[ Y_i = A_i+B_i; \]
$Y_i$ is essentially $A_i$, but with a small perturbation given by $B_i$ to ensure that all these variables are almost surely distinct.

On one hand, we have
\begin{align*}
	\aq_1(Y_1,\dots, Y_n) & = \Pb \left [ Y_1 \leq \bY \leq Y_{2}, \dots, Y_{n} \right]            \\
	                      & \geq \Pb\big[ A_1 =-1, A_2=\cdots=A_n=0 \big]                          \\
	                      & = \frac{1}{n} \big(1-\frac{2}{n} \big)^{n-1} \approx \frac{e^{-2}}{n}.
\end{align*}
On the other hand, we have
\begin{align*}
	\aq_{\lfloor n/2\rfloor}(Y_1,\dots, Y_n) & = \Pb \left [ Y_1,\dots, Y_{\lfloor n/2\rfloor} \leq \bY \leq Y_{\lfloor n/2\rfloor +1}, \dots, Y_{n} \right],                                                          \\
	                                         & =  (\lfloor n/2\rfloor)! (\lceil n/2\rceil)! \Pb\left[ Y_1 \leq \dots \leq Y_{\lfloor n/2\rfloor} \leq \bY \leq Y_{\lfloor n/2\rfloor +1} \leq \dots \leq Y_{n} \right] \\
	                                         & \leq (\lfloor n/2\rfloor)! (\lceil n/2\rceil)! \Pb\left[ Y_1 \leq \dots \leq Y_{\lfloor n/2\rfloor} \leq Y_{\lfloor n/2\rfloor +1} \leq \dots \leq Y_{n} \right]        \\
	                                         & = \frac{1}{\binom{n}{\lfloor n/2 \rfloor}} \approx \frac{\sqrt{n \pi/2}}{2^n}.
\end{align*}
where the second equality holds because the $Y_i$ are almost surely distinct.Hence, we have $\aq_1 > \aq_{\lfloor n/2\rfloor}$ for sufficiently large $n$, so the conclusion of Theorem~\ref{thm:LC} does not hold for these random variables.

Next, we justify the assumptions that Theorem~\ref{thm:CNA} relies on. Pitt~\cite{Pitt} showed that, for jointly normal random variables $Z_1,\dots, Z_m$ with mean $0$, positive correlations $\Eb[Z_i Z_j] \geq 0$ for all $i,j \in [m]$ imply positive associations $Z_i \uparrow Z_j$ for all $i,j \in [m]$. In particular, if condition~\ref{eq:P} is met, then $\bZ \uparrow Z_i$ for all $i\in [m]$. Given this, it is then natural to ask if condition~\ref{eq:P} alone is already sufficient to imply
the conditional positive associations guaranteed by Theorem~\ref{thm:CNA}. However, the example below shows that~\ref{eq:P} alone is not sufficient.

Let $Y_1$ and $Y_2$ be i.i.d.\ standard normal random variables, and set $Z_1= 2Y_1-Y_2$, $Z_2 = Y_1+2Y_2$, and $Z_3=Y_1$.
The covariance matrix $\Sigma=(\Eb\big[Z_i Z_j\big])_{i,j \in [m]}$ of $Z_1$, $Z_2$ and $Z_3$ is given by
\begin{align*}
	\Sigma = \
	\begin{bmatrix}
		5 & 0 & 2 \\
		0 & 5 & 1 \\
		2 & 1 & 1
	\end{bmatrix}.
\end{align*}
All entries of $\Sigma$ are nonnegative, and in particular \ref{eq:P} is satisfied. Now, let $C$ be the event $\{ Z_3 \in [0, \varepsilon] \}$ for some $\varepsilon <{1}/{3}$. Then we have
\begin{align*}
	\Pb\left[ \{\bZ\geq {1}/{3}\} \cap \{Z_1 \geq 1\} \,\vert\, C \right] = \Pb\left[ 1-4Y_1 \leq Y_2 \leq 2Y_1-1 \,\vert\, Y_1 \in [0,\varepsilon] \right] = 0.
\end{align*}
On the other hand,
\begin{align*}
	\Pb\left[ \{\bZ\geq {1}/{3}\} \,\vert\, C \right] \Pb\left[ Z_1\geq 1 \,\vert\, C \right] > 0,
\end{align*}
telling us that the conditional positive association of $\bZ$ and $Z_1$ given $C$ does not hold.

\section {Further directions}\label {sec:sto-conv}
%\com{SH}{I change the title as suggested by Gil.}

Here, we discuss a number of directions for further research related to the results of this paper.

\subsection*{Random points in a prescribed set} The Gaussian model adopted in this paper is convenient for the analytic and empirical study of high-dimensional extensions for the Sylvester four-point question. However, it would be quite interesting to investigate the arguably more common setting for Sylvester's question.

\begin{prob}\label{prob:model-K} Extend Theorem~\ref {thm:ULC} (or its weaker corollaries) to the model of points drawn uniformly at random from a fixed $d$-dimensional convex set $K$.
\end {prob}

It would also be interesting to prove a central limit theorem (like Theorem~\ref{thm:CLT}) for this model.

\subsection*{Radon partitions of larger sets} For fixed integers $d,n \in \mathbb{N}$, consider the probabilities $\ap'_k(n,d)$ that, for a sequence of $n$ random Gaussian points $\Xb_1,\dots,\Xb_n$ in $\mathbb R^{d}$, we have
\[\textnormal{Conv} \{\Xb_1,\dots, \Xb_k\} \cap \textnormal{Conv} \{\Xb_{k+1},\dots, \Xb_n\} \ne \emptyset.\]
Note that, when $n=d+2$, we have
\[ \ap_k = \frac{\ap'_k(n,d)}{2 \binom{n}{k}},\]
for each $k \in \{0,\dots,n\}$, where $p_k=p_k^{(n)}$ is as defined in~\eqref{eq:p_k}.

We conjecture the following generalization of Theorem~\ref{thm:ULC}.
\begin{conj}\label{conj:White}
For all $n \ge d+2$,
\begin{enumerate}
	\item\label{it:White-1} $\ap'_1(n,d) \le \ap'_2(n,d) \le \cdots \le \ap'_{\lfloor n/2 \rfloor}(n,d)$, and moreover,
	\item\label{it:White-2} the sequence $(\ap'_1(n,d),\dots, \ap'_n(n,d))$ is log-concave.
\end{enumerate}
\end {conj}

The first claim in Conjecture~\ref{conj:White} was proposed in~\citep{White}, and it is supported by extensive computational evidence. Theorem~\ref{thm:ULC} establishes Conjecture~\ref{conj:White} when $n = d + 2$, lending further support for the conjecture.

These probabilities \( \ap_k'(n,d) \) also admit the following interpretation in terms of Youden's demon problem:
suppose one has \( r \) Gaussian vectors in \( \rr^n \), where \( r = n - d - 1 \). Then \( \ap_k'(n,d) \) is the probability that these vectors positively span a vector \( V \in \mathbb{R}^n \) such that the first \( k \) coordinates of \( V \) are all at least the mean of its entries, and the last \( n - k \) coordinates are all at most the mean. The equivalence between these two interpretations of \( \ap'_k(n,d) \) follows from the same argument as in Lemma~\ref{lem:iso}.

\subsection* {Counting Tverberg partitions} Recall that Tverberg's theorem asserts that any $(d+1)(r-1)+1$ points in $\rr^d$ can be partitioned into $r$ sets whose convex hulls intersect; such a partition is called a \defnb{Tverberg partition} of order $r$.
A long-standing conjecture due to Sierksma~\citep{Sie79} on the minimal number of Tverberg partitions asserts the following.

\begin{conj}\label{conj:Sierksma}
The minimum number $m(d,r)$ of (unordered) Tverberg partitions of a set of $ (d+1)(r-1)+1$ points in $\rr^d$ is $ ((r-1)!)^d$.
\end {conj}
We note that there are quite a few different configurations for which this minimum is attained; see~\cite{Whi17,BLN17}, for example.

We can also ask about configurations with the maximum number of Tverberg partitions, though now, we need to assume that the points are \emph{generic}. We raise the following two closely related problems.

\begin{prob}
What is the maximum number $M(d,r)$ of Tverberg partitions of a generic set of $ (d+1)(r-1)+1$ points in $ \mathbb R^d$? What is the maximum number $P(d,r)$ of partitions of a generic set of $ d(r-1)+1$ points in $ \mathbb R^d$ into $ r$ parts whose positive hulls have a point in common?
\end {prob}
Let us note that for $d(r-1)+1$ points in $ \mathbb R^d$, it is not always the case that there exists a partition of this set into $r$ parts --- such a partition is called a \defnb{conic Tverberg partition} --- whose positive hulls share a point. It is, however, equivalent to Tverberg's theorem that such a partition exists if the points are contained in an open half space; see e.g.~\cite{Rou01} for a proof.

One may analyse Gaussian random sets to establish reasonable lower bounds for $P(d,r)$ as follows.
\begin{thm}
For fixed $r \in \mathbb{N}$ and large $d$, we have
\[
	P(d,r) \gtrsim \frac {(r/2)^m}{r!},
\]
where $m=d(r-1)+1$.
\end {thm}
\begin{proof}
Let $N(d,r)$ denote the number of (unordered) partitions $\{J_1,\dots, J_r\}$ of $[m]$ into $r$ parts such that $1 \le |J_k| \le d$ for all $k \in [d]$. Note that, for large \( d \), \( N(d,r) \) is asymptotically about \( r^m / r! \), i.e., the number of ways to partition a set of size \( m \) into \( r \) unordered parts, also known as the \defnb{Stirling number of the second kind}.

Let $\Xb_1,\dots, \Xb_m$ be be independent Gaussian vectors in $ \mathbb R^d$. We claim that, almost surely, for every such partition $ J_1,\dots, J_r$, there are signs $s_i \in \{-1, 1\}$ and real numbers $\lambda_i>0$ ($i \in [m]$) such that
\begin{equation}\label{eq:Gil}
	\text{Conv} \{\lambda_is(i) \Xb_i: i \in J_1\} \cap \cdots \cap \text{Conv} \{\lambda_i s(i) \Xb_i: i \in J_r \} \neq \varnothing.
\end{equation}
Indeed, the linear span $\langle\{\Xb_i: i \in J_k\}\rangle$ has codimension $d-|J_k|$, the sum of these codimensions is $rd-\sum |J_k|=d-1<d$, so the intersection of all these linear spans is at least one-dimensional, and therefore nonempty. Pick a unit vector in this intersection (uniformly at random), and then obtain $s_i$ and $\lambda_i>0$ for $i \in J_k$ by writing this vector as a linear combination of $\{\Xb_i: i \in J_k\}$.

Next, the probability for~\eqref{eq:Gil} to hold is the same for all choices of signs. Indeed, this follows from the fact that the normal distribution is invariant under sign-flips. It follows that the expected number of partitions where~\eqref{eq:Gil} holds with the all-one signs is at least $ N(d,r)/2^m$, and of course, in such a case, the positive hulls of the parts of such a partition intersect. We conclude that $ P(d,r) \ge N(d,r)/2^m$. When $ d$ is large, $ N(d,r)$ is asymptotically $r^m/r!$, and the claim follows.
\end {proof}

The argument above actually shows that $P(d,r) \gtrsim \ ({r^{m}}/{r!}) \cdot ({1}/{Z})$, where $Z$ is the probability that $d(r-1)+1$ Gaussian points in $\mathbb R^d$ admit a conic Tverberg partition; estimating $Z$ is a problem of independent interest and it seems possible that $Z$ behaves like $2^{-d+o(1)}$ for large $d$.

Finally, it seems reasonable to expect that $P(d+1,r) \lesssim M(d,r)$, though we do not have a  proper proof of such a relationship. Nevertheless, our belief is guided by the following intuition. Let $\Xb_1,\dots, \Xb_m \in \rr^{d+1}$ be a set of $m = (d+1)(r-1) + 1$ points with the maximum number $P(d+1,r)$ of conic Tverberg partitions, and suppose that the origin is \emph{not contained} in the convex hull of ${\Xb_1,\dots, \Xb_m}$. Then there exists a hyperplane such that all these points lie on the right of this hyperplane. Projecting ${\Xb_1,\dots, \Xb_m}$ onto this hyperplane gives us a $d$-dimensional set ${\Xb_1',\cdots, \Xb_m'}$,
and note that any conic Tverberg $r$-partition of $\Xb_1,\ldots, \Xb_m$ corresponds to a Tverberg $r$-partition of  $\Xb_1',\ldots, \Xb_m'$.
This implies that $P(d+1,r) \leq M(d,r)$ provided there exists a set with the right number of conic Tverberg partitions that does not contain the origin in its convex hull.
%\com{SH}{I add one more sentence here to make things clearer, just in case}.

If it is in fact the case that $P(d+1,r) \lesssim M(d,r)$, then this would imply that $ M(d,r) \gtrsim {(r/2)^m}/{r!}$, where $m = (d+1)(r-1) + 1$, which should be contrasted with the bound of $m(d,r) \lesssim (r/e)^m$ implied by Sierksma's Conjecture (and Stirling's formula).

\begin{figure}[h!]
	\centering
	\includegraphics[scale=0.43]{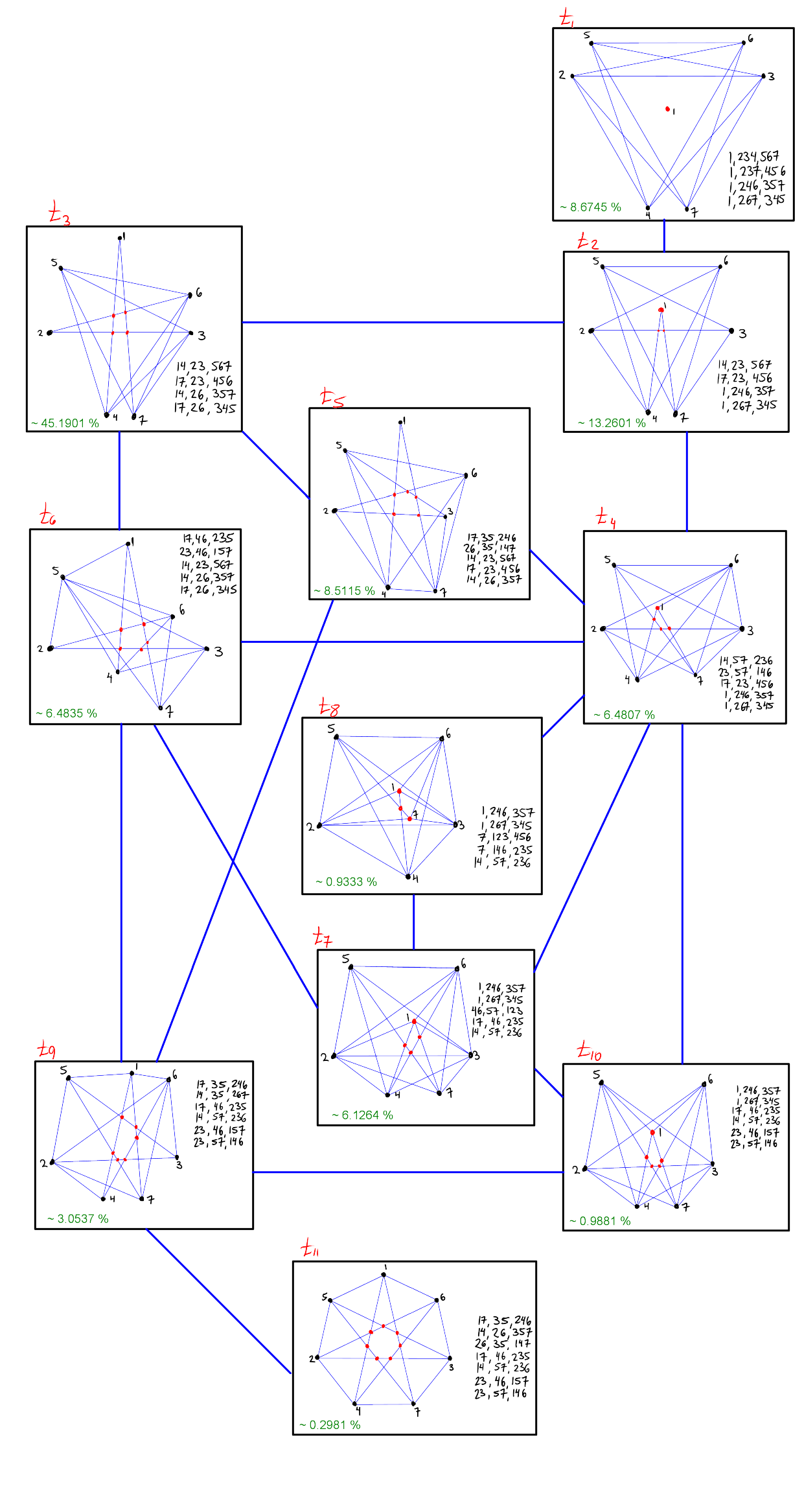}
	\caption{Eleven types of configurations for Tverberg partitions; the line segments between configurations correspond to B\'ar\'any's operations from~\cite{Bar23}.}
	\label{fig:7}
\end{figure}

In the context of counting Tverberg partitions, some numerical results are worth mentioning. We used simulation with Gaussian random points to study the possible Tverberg partitions of order three for seven points in generic position in the plane. The (potentially exhaustive) list of eleven configurations that we found is summarised in Figure~\ref{fig:7}. In these simulations, the minimum number of Tverberg partitions that we encountered was four (in agreement with Sierksma's conjecture), and the maximum was seven (as seen by the vertices of a regular 7-gon).

\subsection*{Reay's Conjecture} Micha Perles~\citep{MPerles} proposed using the probabilistic method to show that,
for fixed $r \in \mathbb{N}$ and $d$ large enough, every set of $2d+2$ points in $\mathbb R^d$ can be partitioned into $r$ sets --- such a partition is called a \defnb{Reay partition} --- whose convex hulls \emph{pairwise} intersect. If true, this would refute a conjecture of Reay~\citep{Rea79}. Some modest evidence for such a claim comes from simulations in~\citep{White} which suggest that for six random Gaussian points in the plane, the probability that a Reay partition exists is $\approx 0.427$.

\section*{Acknowledgements}
We would like to thank Jeff Kahn for initiating discussion between the authors, and Xiaoyu He, Igor Pak, and Sam Spiro for some insightful conversations.

\bibliographystyle{amsalpha}
\bibliography{radon_partitions}

\medskip

\appendix

\section{Proof of Theorem~\ref{thm:CLT}}	\label{sec:proof-CLT}

%	Let $n:=d+2$.
We will prove the theorem by computing the moment sequence of $S_n$,
and we will start with the following lemma.
		Recall that $Y_1,\ldots, Y_n$ are i.i.d. standard normal random variables, and  $\bY:=\tfrac{Y_1+\ldots +Y_n}{n}$.
	Let $D_i$ $(i \in [n])$ be the random variable 
	\[ D_i \ \ := \ \ \mathbbm{1}\{ Y_i \leq \bY\} \. - \. \mathbbm{1}\{ Y_i \geq \bY\} \ \ = \ \ 2 \.  \mathbbm{1}\{ Y_i \leq \bY\} \. - \. 1. \]
	For $k \geq 1$, let   \. $k!! \ts :=  \ts \prod_{i=0}^{\lceil k/2\rceil-1} (k-2i)$ \. be the \defnb{double factorial} of $k$. Note that, $(2k-1)!! = \frac{(2k)!}{k!2^k}.$
	
	\smallskip
	
	\begin{lemma}\label{lem:D_i}
		Fix \. $\ell\geq 1$\.. Then, as $n \to \infty$,
		\begin{align*}
			\Eb[ D_1 \cdots D_\ell] \ = \ 
			\begin{cases}
				\big(\frac{-2}{\pi n}\big)^{\ell/2} \. {(\ell-1)!!} \  + \ O(n^{-\ell/2-1}) & \text{ if $\ell$ is even},\\[3 pt]
				0 & \text{ if $\ell$ is odd.}
			\end{cases}  
		\end{align*}  
	\end{lemma}
	
	\smallskip

	\begin{proof}
		The case when $\ell$ is odd follows from the fact that \. $(D_1,\ldots, D_\ell)$ \. has the same distribution as 
		\. $(-D_1,\ldots, -D_\ell)$\..
		We now consider the case when $\ell$ is even.

Let \. $Z_1,\ldots Z_\ell$  \. be jointly normal random variables given by 
\[  Z_i \ := \   \bY- Y_i \qquad (i \in [\ell]). \]
Note that these random variables satisfy
\[ \Eb[Z_i^2] \ = \  1-\frac{1}{n}, \qquad  \Eb[Z_i Z_j] \ = \  -\frac{1}{n}  \qquad \text{ for distinct } \ i,j \in [\ell].\]
It follows that the inverse \ts $Q$ \ts of the covariance matrix of $Z_1,\ldots, Z_\ell$
is equal to
\[Q_{ii} \ = \  1+ \frac{1}{n-\ell}, \qquad Q_{ij} \ = \  \frac{1}{n-\ell} \qquad \text{ for distinct } \ i,j \in [\ell]. \]
It then follows that 
\begin{align*}
	& \Eb[ D_1 \cdots D_\ell] \  \ = \ \ \Eb\big[ \prod_{i=1}^\ell \text{sign}(Z_i)\big]\\
	&  = \ \  \frac{1}{(2\pi)^{\ell/2} \sqrt{1-\tfrac{\ell}{n}}} \.   \int_{[-\infty,\infty]^{\ell}} \. \prod_{i=1}^\ell \text{sign}(z_i) \. \exp\big(-\tfrac{1}{2}(1+\tfrac{1}{n-\ell}) z_i^2\big) \.
	\prod_{i\neq j}  \exp\big(-\tfrac{z_iz_j}{n-\ell}\big) \. dz_1 \cdots dz_\ell\\
	& = \ \  \int_{[-\infty,\infty]^{\ell}}  f(z_1,\ldots, z_n) \sum_{s=0}^{\infty} \frac{1}{s!} \. \big(\frac{-1}{n-\ell}\big)^{s}  \. \big( \sum_{i \neq j} z_i z_j \big)^s \ dz_1 \cdots dz_\ell,
\end{align*}
where 
\[   f(z_1,\ldots, z_n) \ := \  \frac{1}{(2\pi)^{\ell/2} \sqrt{1-\tfrac{\ell}{n}}} \.  \prod_{i=1}^\ell \text{sign}(z_i) \. \exp\big(-\tfrac{1}{2}(1+\tfrac{1}{n-\ell}) z_i^2\big).  \]

We will now break the sum up based on the value of $s.$ For large $s,$ it now follows from a standard calculation that 
\begin{align*}
	 \int_{[-\infty,\infty]^{\ell}}  f(z_1,\ldots, z_n) \sum_{\blue{s> \ell/2}} \frac{1}{s!} \. \big(\frac{-1}{n-\ell}\big)^{s}  \. \big( \sum_{i \neq j} z_i z_j \big)^s \ dz_1 \cdots dz_\ell  \ \ = \ \  O(n^{-\ell/2-1}).
\end{align*}
Also note that, for any  $a_1,\ldots, a_\ell$ such that $a_j=0$ for some $j \in [\ell]$, we have 
\begin{align*}
& {(2\pi)^{\ell/2} \sqrt{1-\tfrac{\ell}{n}}}	   \int_{[-\infty,\infty]^{\ell}}  f(z_1,\ldots, z_n) \. \blue{z_1^{a_1} \ldots z_\ell^{a_\ell}}   \ dz_1 \cdots dz_\ell \\
	  &= \ \  \bigg(\prod_{i \neq j} \int_{-\infty}^{\infty} \. \text{sign}(z_i) \. \exp\big(-\tfrac{1}{2}(1+\tfrac{1}{n-\ell}) z_i^2\big) \.
	  z_i^{a_i}  \. dz_i  \bigg)  \ \bigg(  \int_{-\infty}^{\infty} \. \text{sign}(z_j) \. \exp\big(-\tfrac{1}{2}(1+\tfrac{1}{n-\ell}) z_j^2\big)   \. dz_j \bigg) \\
	  &= \ \  0,
\end{align*}
since the normal random variables are symmetric around 0. In particular, this is true for all $s <\ell/2$.

Combining the calculations above, we then have 
\begin{align*}
	& \Eb[ D_1 \cdots D_\ell] \\
	& = \ \   \int_{[-\infty,\infty]^{\ell}}  f(z_1,\ldots, z_n) \.  \frac{1}{(\ell/2)!} \.  \left(\frac{-1}{n-\ell}\right)^{\ell/2}  \. \left( \sum_{i \neq j} z_i z_j \right)^{\ell/2} \ dz_1 \cdots dz_\ell \. + \.  O(n^{-\ell/2-1})   \\[3 pt]
		& = \ \   \int_{[-\infty,\infty]^{\ell}}  f(z_1,\ldots, z_n) \.  \frac{1}{(\ell/2)!} \.  \left(\frac{-1}{n-\ell}\right)^{\ell/2}  \.  \frac{\ell!}{2^{\ell/2}} \.  z_1 \cdots z_\ell \ dz_1 \cdots dz_\ell \. + \.  O(n^{-\ell/2-1})   \\[3 pt]
	&  = \ \left(\frac{2}{\pi}\right)^{\ell/2}  \frac{ (\ell-1)!!}{\sqrt{1-\tfrac{\ell}{n}}}  \left(\frac{-1}{n-\ell}\right)^{\ell/2} \int_{[0,\infty]^{\ell}} \prod_{i=1}^\ell  |z_i| \exp\left(-\tfrac{1}{2}(1+\tfrac{1}{n-\ell}) z_i^2\right)  	 dz_1 \ldots dz_\ell
 %    \\[3 pt]
	% & \qquad 
    \.+\. O(n^{-\ell/2-1}) \\[3 pt]
	& = \ \  \left(\frac{2}{\pi}\right)^{\ell/2}  \frac{ (\ell-1)!!}{\sqrt{1-\tfrac{\ell}{n}}}  \. \left(\frac{-1}{n-\ell}\right)^{\ell/2} \.   \left(1+\frac{1}{n-\ell} \right)^{-\ell}
	\. +  \. O(n^{-\ell/2-1}) \\[3 pt]
	& = \ \ \left(\frac{-2}{ \pi n}\right)^{\ell/2} \. {(\ell-1)!!} \  + \ O(n^{-\ell/2-1}),
\end{align*}
as desired.
	\end{proof}
	
	\smallskip
	
\begin{lemma}\label{lem:T_i}
	Fix $\ell \geq 1$. Then
	\begin{equation*}
	\lim_{n \to \infty} \Eb\bigg[	\bigg(\frac{D_1+\ldots+D_n}{\sqrt{n}}\bigg)^{\ell} \bigg] \ \ = \ \ 
	\begin{cases}
(\ell-1)!! \. (1-\tfrac{2}{\pi})^{\ell/2} & \text{ if $\ell$ is even},\\
0 & \text{ if $\ell$ is odd}.
	\end{cases}
	\end{equation*}
\end{lemma}
	
	\smallskip
	
	\begin{proof}
		The case when $\ell$ is odd follows directly from the odd case of Lemma~\ref{lem:D_i}.
		We now consider the case when $\ell$ is even.
        The sum of the terms of the form, \. $n^{-\ell/2}\Eb\big[D_1^{a_1}\ldots D_n^{a_n}\big]$ \. where \ts $a_i \geq 3$ \ts for some $i\in [n]$, does not exceed $O(n^{-1})$, so     
% 		It follows from Lemma~\ref{lem:D_i} that 
% 		\begin{align*}
% &\Eb\bigg[	\bigg(\frac{D_1+\ldots+D_n}{\sqrt{n}}\bigg)^{\ell} \bigg] \\[3 pt] 
% &= \ \    n^{-\ell/2}\sum_{s=0}^{\ell/2} \binom{n}{2s} \. \binom{n-2s}{\ell/2-s}  \. \frac{\ell!}{2^{\ell/2-s}} \. \Eb[D_1\ldots D_{2s}  \. (D_{2s+1})^2 \. \ldots (D_{\ell/2+s})^2] \ + \ O(n^{-1}),
% 		\end{align*}
% because the sum of the terms of the form, \. $n^{-\ell/2}\Eb\big[D_1^{a_1}\ldots D_n^{a_n}\big]$ \. where \ts $a_i \geq 3$ \ts for some $i\in [n]$, does not exceed $O(n^{-1})$.	
% Continuing this calculation, 
it follows from Lemma~\ref{lem:D_i} that
			\begin{align*}
		&\Eb\bigg[	\bigg(\frac{D_1+\ldots+D_n}{\sqrt{n}}\bigg)^{\ell} \bigg] \\[3 pt] 
		&= \ \    n^{-\ell/2}\sum_{s=0}^{\ell/2} \binom{n}{2s} \. \binom{n-2s}{\ell/2-s}  \. \frac{\ell!}{2^{\ell/2-s}} \. \Eb[D_1\ldots D_{2s}  \. (D_{2s+1})^2 \. \ldots (D_{\ell/2+s})^2] \ + \ O(n^{-1})\\[3 pt]
		&= \ \    n^{-\ell/2} \sum_{s=0}^{\ell/2} \binom{n}{2s} \. \binom{n-2s}{\ell/2-s}  \. \frac{\ell!}{2^{\ell/2-s}} \. \big(\frac{-2}{\pi n}\big)^{s} \. {(2s-1)!!}      \ + \ O(n^{-1})\\[3 pt]
		&= \ \     \sum_{s=0}^{\ell/2} \frac{1}{(2s)!} \. \frac{1}{(\ell/2-s)!}  \. \frac{\ell!}{2^{\ell/2-s}} \. \big(\frac{-2}{\pi}\big)^{s} \. {(2s-1)!!}      \ + \ O(n^{-1})\\[3 pt]
		&= \ \     \frac{\ell!}{2^{\ell/2}}  \sum_{s=0}^{\ell/2}  \big(\frac{-2}{\pi}\big)^{s} \.   \frac{1}{s!(\ell/2-s)!}      \ + \ O(n^{-1})\\[3 pt]
		&= \ \     \frac{\ell!}{2^{\ell/2}(\ell/2)!} \bigg(1-\frac{2}{\pi}\bigg)^{\ell/2}  \ + \ O(n^{-1})\\[3 pt]
		&= \ \    (\ell-1)!! \. (1-\tfrac{2}{\pi})^{\ell/2} \ + \ O(n^{-1}),
	\end{align*}
	and the lemma now follows.
	\end{proof}

	\smallskip
	
	\begin{proof}[Proof of Theorem~\ref{thm:CLT}]
			It  follows from Lemma~\ref{lem:iso} 
		that 
		\[  {S_n- \tfrac{n}{2}} \ \  = \  \  \sum_{i=1}^n  \left(  \mathbbm{1}\{ Y_i \leq \bY\}  - \tfrac{1}{2}\right) \ \ = \ \ \frac{1}{2} \sum_{i=1}^n D_i. \]
		On the other hand, it follows from Lemma~\ref{lem:T_i} that
		\[ \frac{D_1+\ldots+D_n}{\sqrt{n}} \ \  \text{ converges in distribution to  normal RV with mean $0$ and variance $1-\frac{2}{\pi}$} .   \]
%		\com{SH}{Can we drop the assumption that $Y_i$ are normal random variables?}\reply{NTS}{The current proof method seems to rely on normality }
		The theorem now follows by combining the two observations above.
	\end{proof}

\end{document}